\documentclass[11pt]{article}

\usepackage{amsmath, graphicx}
\usepackage[british]{babel}
\usepackage{amsthm, amsmath, amssymb, mathtools}
\usepackage{enumerate}
\usepackage{graphicx}
\usepackage{xfrac}
\usepackage[a4paper, width = 15cm]{geometry}
\usepackage{hyperref}

\DeclareGraphicsExtensions{.pdf,.png,.eps}

\newtheorem{theorem}{Theorem}
\newtheorem{lemma}[theorem]{Lemma}

\newtheorem*{remark}{Remark}
\theoremstyle{remark}
\newtheorem*{claim*}{Claim}

\DeclarePairedDelimiter\ceil{\lceil}{\rceil}

\newcommand{\N}{\ensuremath{\mathbb{N}}}

\newcommand{\cd}{\ensuremath{{\rm deg}}}
\renewcommand{\epsilon}{\varepsilon}

\title{Interval colorings of graphs -- coordinated and unstable no-wait schedules}

\author{Maria Axenovich\\
\small Department of Mathematics\\[-0.8ex]
\small Karlsruhe Institute of Technology\\[-0.8ex] 
\small Germany\\
\small\href{mailto:maria.aksenovich@kit.edu}{\tt maria.aksenovich@kit.edu} \\
\and
Michael Zheng\\
\small Department of Mathematics\\[-0.8ex]
\small Karlsruhe Institute of Technology\\[-0.8ex] 
\small Germany\\
\small\href{mailto:mxxz20@gmail.com}{\tt mxxz20@gmail.com}}

\begin{document}

\maketitle

\begin{abstract}
A proper edge-coloring of a graph is  an interval coloring if  the labels on the edges incident to any vertex form an interval of consecutive integers.  
Interval thickness $\theta(G)$ of a graph $G$ is the smallest number of  interval colorable graphs edge-decomposing G.  We prove that $\theta(G) =o(n)$ for any graph $G$ on $n$ vertices. 
This improves the previously known bound of $2\lceil n/5 \rceil$, see Asratian, Casselgren, and Petrosyan \cite{ACP}.
While we do not have a single example of a graph with interval thickness strictly greater than $2$, we construct bipartite graphs whose interval spectrum has arbitrarily many arbitrarily large gaps. Here, an interval spectrum of a graph is the set of all integers $t$ such that the graph has an interval coloring using $t$ colors.

Interval colorings of bipartite graphs naturally correspond to no-wait schedules, say for parent-teacher conferences, where a conversation between any teacher and any parent lasts the same amount of time. Our results imply that any such conference with $n$ participants can be coordinated in $o(n)$ no-wait periods.  In addition, we show that for any integers  $t$ and $T$, $t<T$,   there is a set of pairs of  parents and teachers wanting to talk to each other, such that any no-wait schedules are unstable --  they could last $t$ hours and could last $T$ hours, but there is no possible no-wait schedule lasting $x$ hours if $t<x<T$. 
\end{abstract}

\section{Introduction}
Astarian and Kamalian \cite{AK} introduced the notion of interval colorability of graphs. We say that a graph $G=(V,E)$ is {\it interval colorable} if there is an edge-coloring $c\colon E \rightarrow \mathbb{Z}$ such that for any vertex $x$, the multiset of colors incidents to $x$, i.e., $\{ c(xy): xy\in E\}$ forms a set of consecutive integers, in other words an {\it interval of integers}. The respective coloring is called an {\it interval coloring}. In particular, an interval coloring is a proper coloring, i.e., there are no two adjacent edges having the same color. \\

  Interval colorings are applied in scheduling -- for example in case of teacher-parent conferences or machine-jobs assignments. In the former case one wants to schedule meetings between a parent and a teacher for given parent-teacher pairs such that each such meeting lasts the same amount of time and there is no waiting time between the meetings for any of the parents and any of the teachers.  \\

Interval colorable graphs include all trees. In addition, all regular bipartite graphs are interval colorable since by K\H{o}nig's theorem they are edge decomposable into perfect matchings. On the other hand, any graph of Class $2$ is not interval colorable, where a graph is of Class $2$ if its edge-chromatic number is greater than its maximum degree, $\Delta(G)$. For example a triangle is such a graph. Indeed, otherwise considering the labels in an interval coloring modulo $\Delta(G)$ gives a proper edge coloring using at most $\Delta(G)$ colors. Interval colorings for special classes of graphs and related problems were considered, see for example \cite{ACP, As,  ADH,  Ax, BD, BD1, BDJZ, CP, CT,   GK, GK1,  H, HLT,   PKh, P,  YL, Z}.\\

Let $c(G) = \{c(e): e\in E(G)\}$ be the set of colors used on $G$ by a coloring $c$. 
It is easy to see that for a connected graph $G$, and an interval coloring $c$, $c(G)$ is a set of consecutive integers. Here, we shall assume that all considered graphs are connected. Moreover, we assume that all objects considered are finite.\\

If there is an interval coloring of a graph $G$ such that $|c(G)|=t$, we say that $G$ is $t$-{\it interval colorable}. 
Let the {\it interval spectrum} of $G$, denoted by $S(G)$ be the set of all integers $t$ such that $G$ is $t$-interval colorable. Note that $S(G)$ might be empty.
The {\it  interval thickness} of a graph $G$, denoted $\theta(G)$,  is the smallest integer $k$ such that the graph can be edge-decomposed into $k$ interval colorable graphs. In the language of parent-teacher conferences, having an interval thickness of the respective graph equal to $x$ implies that one can schedule the conference in $x$ days with no waiting time for anyone during any of these $x$ days. Let $\theta(n)$ be the largest interval thickness of an $n$-vertex graph. Interval thickness was considered for several special classes of graphs and bounded in terms of various graph parameters, \cite{ACP}.
Most notably $\theta(G)\leq \gamma(G)$, where $\gamma(G)$ is the arboricity of $G$, i.e., the minimum number of forests edge-decomposing $G$.
The following result gives best known bounds on $\theta(n)$.

 \begin{theorem}[Asratian, Casselgren, Petrosyan \cite{ACP}]\label{basic}
 For any integer $n\geq 3$, $2\leq \theta(n) \leq 2 \lceil n/5 \rceil$.
 \end{theorem}  

Here, we improve the upper bound:
\begin{theorem}\label{reg-bound}
$\theta(n)=o(n)$.
\end{theorem}

To prove this result we employ the Regularity Lemma of Szemer\'edi, a version presented in Diestel \cite{D}, and a result by Alon, R\"odl, and Ruci\'nski \cite{ARR} showing an existence of dense regular subgraphs in $\epsilon$-regular pairs.\\

In addition, we show that the spectrum could have large gaps of large sizes. Here, a {\it gap} of a set of integers $S$ is a maximal non-empty set $X$ of consecutive integers, such that $X\cap S= \emptyset$, 
$\min S <\min X$ and $\max S>\max X$. For example, a set $\{2, 3, 6, 7\}$ has one gap $\{4, 5\}$ of size $2$.

\begin{theorem}\label{spectral-gaps}
For any natural numbers $k$ and $d$ there is a graph $G$ such that the spectrum $S(G)$ has exactly  $k$  gaps of size at least $d$ each.
\end{theorem}

This theorem is proved by giving an explicit construction of such a graph that in turn is built of parts from a construction by Sevastianov \cite{S}. In the language of parent-teacher conferences, this result implies for example that there could be such a set of parent-teacher  pairs  willing to talk to each other so that one can schedule an optimal  no-wait conference lasting $5$ hours, but if the school secretary doesn't manage to find an optimal scheduling, the only other option for a no-wait conference would require at least $105$ hours.\\

We shall  give necessary definitions and preliminary results for the upper bound on interval thickness  in Section \ref{observations} and for the gaps in the interval spectrum in Section \ref{constructions}.  The main results are proved in Section \ref{proofs}.  For some results in this paper, see also a bachelor thesis of the second author, M. Zheng,  \cite{Z}.

%%%%%%%%%%%%%%%%%%%%%%%%%%%%%%%%%%%%%%%%%%%%%%%%%%%%%%%%%%%%%%%%%
%%%%%%%%%%%%%%%%%%%%%%%%%%%%%%%%%%%%%%%%%%%%%%%%%%%%%%%%%%%%%%%%%
                      \section{Definitions and preliminary results} \label{observations}

%%%%%%%%%%%%%%%%%%%%%%%%%%%%%%%%%%%%%%%%%%%%%%%%%%%%%%%%%%%%%%%%%
%%%%%%%%%%%%%%%%%%%%%%%%%%%%%%%%%%%%%%%%%%%%%%%%%%%%%%%%%%%%%%%%%

For standard graph theoretic notions we refer the reader to the book by Diestel \cite{D}. We shall denote the number of vertices and the number of edges  in a graph $G$ by $|G|$ and $\| G\| $ respectively.  We shall need some standard terminology for using the Regularity Lemma. 
	For a graph $G$,  let $X$ and $Y$ be disjoint vertex sets and $\varepsilon > 0$. We define $G[X,Y]$ to be a bipartite graph with parts $X$ and $Y$ containing all edges of $G$ with one endpoint in $X$ and another in $Y$. 
		Let  $\| {X,Y}\| $ to be the number of edges in $G[X,Y]$ and the density $d(X,Y)$ of $(X,Y)$ to be
				$d(X,Y) = \frac{\| {X,Y}\| }{|X\| Y|}.$ Let $\delta(X,Y)=\delta(G[X,Y])$, be the mininum degree of $G$.  For a vertex $x$, we denote the neighbourhood of $x$ by  $N(x)$ and the degree of $x$ by $\cd(x)$.

			 A pair $(X,Y)$ is an {\it $\varepsilon$-regular pair}  in $G$ or more precisely a $(d, \varepsilon)$-regular pair if  $X$ and $Y$ are disjoint vertex sets in $G$ and  $|{d - d(A,B)}| \leq \varepsilon$ for all $A \subseteq X, B \subseteq Y$ with $|A| \geq \varepsilon |{X}|, |B| \geq \varepsilon |Y|$ and $d = d(X,Y)$. 
			 We call an $\varepsilon$-regular pair $(X,Y)$  in $G$ a {\it super $\varepsilon$-regular pair} in $G$ if $|X|=|Y|$ and 
			$$\delta(X,Y) \geq (d(X,Y) - \varepsilon)|X|. $$
			A bipartite graph $G'$ with parts $X$ and $Y$ is called a {\it super $(d,\varepsilon)$-regular graph}  if $(X,Y)$ is super $\varepsilon$-regular pair in $G'$ with density $d$.
			%If $(X,Y)$ is super $\varepsilon$-regular pair in $G$ with density $d$, we call $G[X,Y]$ a {\it super $(d,\varepsilon)$-regular graph} with parts $X$ and $Y$. An {\it $\varepsilon$-regular partition} of the graph $G = (V,E)$ is a partition of the vertex set $V = V_0 \cup V_1 \cup \cdots \cup V_k$ with the following three properties: 
			\begin{enumerate}
				\item $|{V_0}| \leq \varepsilon |{V}|$, 
				\item $|{V_1}| = |{V_2}| = \dots = |{V_k}|$,  
				\item All but at most $\varepsilon k^2$ of the pairs $(V_i, V_j)$ for $1 \leq i < j \leq k$ are $\varepsilon$-regular. 
			\end{enumerate}
	\begin{theorem}[Szemer\'edi's Regularity Lemma \cite{D}]
		For every $\varepsilon > 0$ and every integer $m \in \N$ there is an $M \in \N$ such that every graph of order at least $m$ has an $\varepsilon$-regular partition $V_0 \cup \cdots \cup V_k$ with $m \leq k \leq M$. 
		\label{thm:szemeredi}
	\end{theorem}

	\begin{lemma}[Alon, R\"odl, and Ruci\'nski \cite{ARR}]
		Let $G'$ be a bipartite super $(d,\varepsilon)$-regular graph with parts of size $n$ each and let $d > 2 \varepsilon$. Then $G'$ contains a spanning $k$-regular subgraph, where $k = \lceil{(d- 2\varepsilon) n}\rceil$.  
		\label{ARR}
	\end{lemma}

	The following standard lemma shows that an $\varepsilon$-regular pair contains a large super $\varepsilon$-regular pair.
	\begin{lemma}
		Let $(X,Y)$ be a $(d, \varepsilon)$-regular pair in a graph $G$  with $d > 4\varepsilon$ and $|{X}| = |{Y} |= n$. Then, there are sets $X'\subseteq X,~ Y' \subseteq Y$ such that $|{X'}| = |{Y'} |> (1 - \varepsilon) n$ and $(X', Y')$ is a super $3\varepsilon$-regular pair in $G$ with density $d'$ where $d' \geq d- \varepsilon$.
		\label{super}
	\end{lemma}
	\begin{proof}
		Let 
		\begin{eqnarray*}
		\tilde{X} &= & \{x \in X \colon |{N(x) \cap Y}| \geq (d- \varepsilon) |{Y}|\} \mbox{ and}\\ 
\tilde{Y}&=&\{y \in Y \colon |{N(y) \cap X}| \geq (d- \varepsilon) |{X}|\}.
\end{eqnarray*}

		Note that $|{\tilde{X}}| > (1 - \varepsilon)|{X}|$. 
		Otherwise, let $X_2= X\setminus \tilde{X}$  and observe that any vertex in $X_2$ has less than $(d-\varepsilon)$ neighbours in $Y$, thus 
		$d(X_2, Y) <(d-\epsilon)$, a contradiction to $\varepsilon$-regularity since $|X_2|\geq \varepsilon n$.
		A similar argument holds for $\tilde{Y} $.\\
		
Let $X' \subseteq \tilde{X},$ $ Y' \subseteq \tilde{Y}$ such that $|X'| = |Y'| = \min \{|\tilde{X}|, |\tilde{Y}|\}$. Then $|X'|=|Y'|=n' > (1 - \varepsilon) |X|= (1-\varepsilon )n.$
Note that for $\varepsilon <1/2$ we have $n'>\varepsilon n$. We shall show that $(X', Y')$ satisfies the minimum degree and regularity conditions of a super-regular pair.\\

		Let $d' = d(X', Y')$. By $\varepsilon$-regularity of $(X,Y)$, we have $d' \leq d + \varepsilon.$  Moreover, from the definition of $X'$ we have  $\delta(X', Y') \geq (d - \varepsilon) n - \varepsilon n = (d - 2 \varepsilon) n \geq (d' - 3 \varepsilon) n \geq (d' - 3 \varepsilon) n'.$ 
Now, consider $A \subseteq X', B \subseteq Y',$ such that $ |A| \geq 3\varepsilon  |X'| $ and $|B| \geq 3\varepsilon |{Y'}|$. Observe that 
	$|A| \geq 3 \varepsilon |X'|  > 3 \varepsilon (1- \varepsilon) n >  \varepsilon n.$  Similarly,  $|B|>\varepsilon n$. Then, by $\varepsilon$-regularity of $(X,Y)$,  we obtain  that
$$|{d(X', Y') - d(A, B)}| \leq |{d(X', Y') - d(X,Y)}| + |{d(A,B) - d(X,Y)}|\leq 2 \varepsilon \leq 3 \varepsilon.$$
		Thus, $(X',Y')$ is a $3\varepsilon$-regular pair with density $d- \varepsilon \leq d' \leq d+ \varepsilon$ and minimum degree $\delta(X', Y') \geq (d' - 3 \varepsilon) n $.
	\end{proof}

	\begin{theorem}
		For every $\gamma$, $1 / 2 > \gamma > 0$, there exists $M \in \N$ such that every graph $G$ contains a subgraph $G'$ with $\theta(G') \leq M$ and 
		$\| {G}\|  - \| {G'}\|  \leq \gamma |{G}|^2$. 
		\label{gamma}
	\end{theorem}
	\begin{proof}
		Let $1 / 2 > \gamma > 0$ be arbitrary. Choose $\varepsilon > 0$ sufficiently small and $m \in \N$ sufficiently large such that  $\frac{1}{2m} + \left(11+ \frac{\varepsilon}{2}\right) \varepsilon \leq \gamma.$
		By Szemer\'edi's Regularity Lemma, see Theorem \ref{thm:szemeredi}, there exists $M \in \N$ such that every graph of order at least $m$ has an $\varepsilon$-regular partition $V_0 \cup \dots \cup V_k$ with $m \leq k \leq M$. Now, let $G$ be a graph of order $n \in \N$. If $n < m$, we have that 
		$\theta(G) \leq n \leq M$ using an upper bound in Theorem \ref{basic}. Thus, we may assume that $n \geq m$. By our choice of $M$, we know that $G$ has an $\varepsilon$-regular partition $V_0 \cup \cdots \cup V_k$ with $m \leq k \leq M$. Let $\ell = |{V_1}|$.  We shall define a subgraph $G'$ of $G$ corresponding to regular pairs of sufficiently high density. \\
	
	For each pair $(V_i,V_j)$, $1 \leq i < j \leq k$, we shall define a graph $G_{i,j}$. Let $d_{i,j}$ be the  density of $(V_i, V_j)$.  
	If $d_{i,j} >  7 \varepsilon$  and $(V_i, V_j)$ is an $\varepsilon$-regular pair,  consider $G[V_i,V_j]$ and apply Lemmas \ref{ARR} and \ref{super} to it.  
	Let $G_{i,j}$ be a subgraph of $G[V_i,V_j]$ that is $q_{i,j}$-regular on at least $2(1-\varepsilon)\ell$ vertices and  $q_{i,j} \geq (d_{i,j} - \varepsilon - 2(3\varepsilon))n$. 
Note that since $G_{i,j}$ is bipartite and regular, it is interval colorable. Note also that $G_{i,j}$ contains most of the edges of $G[V_i, V_j]$. We shall make this statement more precise below.
If a pair  $(V_i,V_j)$ is not $\varepsilon$-regular or has density at most $7\varepsilon$, let $G_{i,j}$ be an empty graph. Let  $$G' = \bigcup_{1 \leq i < j \leq k} G_{i,j}. $$

 We shall show first that $\theta(G') \leq M$. For that let $c$ be a proper edge coloring of a complete graph with vertex set $\{1, \ldots, k\}$  using colors from $\{1, \ldots, k\}$ and let, for $s \in \{1, \ldots, k\}$,
			$$G_s = \bigcup_{1 \leq i < j \leq k,~ c(ij) = s} G_{i,j}.$$
		 
		 Since  $G_s$ is the vertex-disjoint union of interval colorable graphs,  $G_s$ is itself interval colorable, for all $s\in \{1, \ldots, k\}$.    
		Since $G'= \bigcup_{1\leq s \leq k} G_s$,  we have that $\theta(G') \leq k \leq M$.\\

		We will bound the number of edges from  $G$  that are not in $G'$. We call a pair $(V_i, V_j)$, $i\neq j$,  nontrivial if $i\neq 0$ and $j\neq 0$.  We have that $\| G\| -\| G'\| = x_1+x_2+x_3+x_4$, where
		\begin{itemize}
			\setlength\itemsep{-0.25em}
			\item $x_1$ is the number of edges in non-$\varepsilon$-regular pairs or with exactly one endpoint in $V_0$, 
			\item $x_2$ is the number of edges induced by $V_i$'s for $0 \leq i \leq k$,
			\item $x_3$ is the number of edges in nontrivial $\varepsilon$-regular pairs with density at most $7 \varepsilon$, and
			\item $x_4$ is the number of edges in nontrivial $\varepsilon$-regular pairs with density greater than $7 \varepsilon$ that are not in $G'$.  
		\end{itemize}
			
		 Note that $\ell=|V_1|=\cdots=|V_k|\leq n/k$ and $|V_0|\leq \varepsilon n$. Moreover, the maximum number of edges in $G[V_i, V_j]$ is at most $(n/k)^2$, for $1 \leq i<j\leq k$.
		Since there are at most $\varepsilon k^2$ nontrivial  pairs that  are non-$\varepsilon$-regular and at most $|V_0\| V(G)-V_0| \leq \varepsilon(1-\varepsilon) n^2$ edges with exactly one endpoint in $V_0$,  we have 
			$$x_1 \leq \varepsilon k^2 \cdot \left({\frac{n}{k}}\right)^2 + \varepsilon n (n - \varepsilon n ) \leq 2 \varepsilon n^2.$$
In addition, $$x_2 \leq \binom{\varepsilon n}{2} + k \binom{n /k}{2} \leq \frac{(\varepsilon n)^2}{2} + \frac{n^2}{2 k}$$ and $$x_3 \leq \binom{k}{2} 7 \varepsilon (n / k)^2 \leq \frac{7}{2} \varepsilon  n^2.$$
Finally, for $x_4$, note that for a pair with parts of size $\ell$ each and with density $d > 7 \varepsilon$, the number of edges that are not in $G'$ is at most 
		$d \cdot \ell^2 - (d - 7 \varepsilon) \cdot (\ell(1 - \varepsilon))^2 \leq  10 \varepsilon l^2 \leq 10 \varepsilon \left({\frac{n}{k}}\right)^2$.
		Thus  $$x_4 \leq \binom{k}{2} 10 \varepsilon (n / k)^2 \leq 5 \varepsilon n^2.$$

		Therefore, 
		\begin{eqnarray*}
			\| G\| -\| G'\|  &= & x_1 + x_2 + x_3 + x_4 \\
			&\leq & \left(2 \varepsilon + \frac{\varepsilon^2 }{2} + \frac{1}{2k} + \frac{7}{2} \varepsilon + 5\varepsilon \right)n^2 \\
			&\leq &\left(\frac{1}{2m} + 11\varepsilon + \frac{\varepsilon^2}{2} \right) n^2 \\
			&\leq &\gamma n^2.
		\end{eqnarray*}
		This concludes the proof.
	\end{proof}

%%%%%%%%%%%%%%%%%%%%%%%%%%%%%%%%%%%%%%%%%%%%%%%%%%%%%%%%%%%%%%%%%
%%%%%%%%%%%%%%%%%%%%%%%%%%%%%%%%%%%%%%%%%%%%%%%%%%%%%%%%%%%%%%%%%	

                   \section{Construction of a graph with a given interval spectrum}	\label{constructions}

%%%%%%%%%%%%%%%%%%%%%%%%%%%%%%%%%%%%%%%%%%%%%%%%%%%%%%%%%%%%%%%%%
%%%%%%%%%%%%%%%%%%%%%%%%%%%%%%%%%%%%%%%%%%%%%%%%%%%%%%%%%%%%%%%%%	

\subsection{Construction and properties of the graph $F(b,T)$}

\begin{figure}[h!]
	\centering
	\includegraphics[width = 0.7 \textwidth]{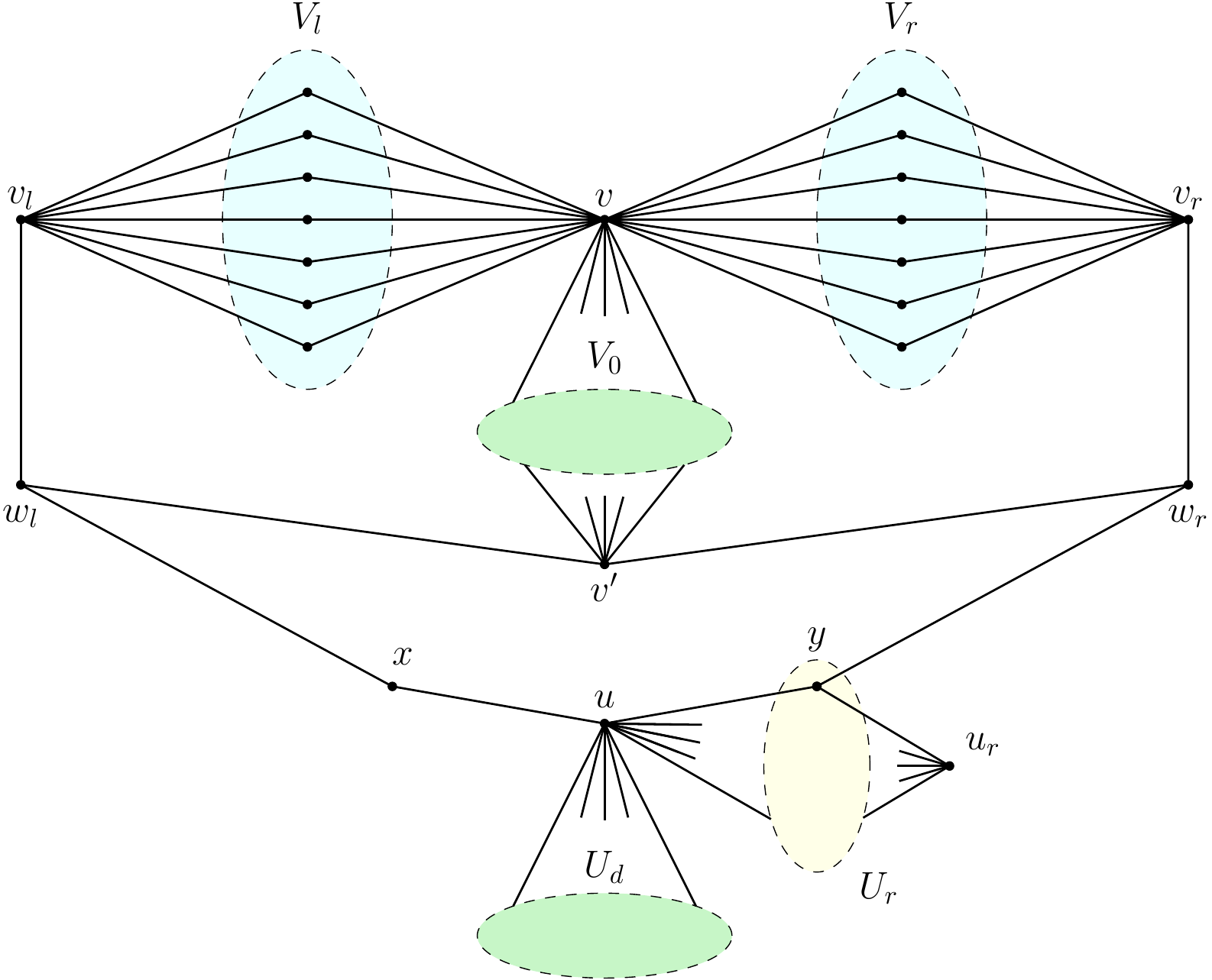}
	\caption{Graph $F(b,T)$}
	\label{fig:F(b,D)}
\end{figure}

	For positive integers $b, D$, where $ b\leq D$, let  $T=D+25$ and the graph $F=F(b,T)$ be formed by a union of five complete bipartite graphs with pairs of parts 
	$(\{v, v'\}, V_0)$, $(\{v, v_r\}, V_r)$, $(\{v, v_l\}, V_l)$, $(\{u, u_r\}, U_r)$, $(\{u\}, U_d)$, as well as additional vertices $w_l, w_r$ and edges $w_lv', $ $w_rv', $ $w_lv_l, $ $w_rv_r, $ $w_lx,$ $ w_ry, $ $xu$, where $y\in U_r$. Here the vertices $x,$ $ v, v', $ $v_l, $ $v_r,$ $ u, $ $u_r, $ $w_l, $ and $w_r$ are distinct and not contained in any of the pairwise disjoint sets $V_0, V_l, V_r, U_l, U_r, $ and $U_d$. 
Moreover $|V_0|= D+12$, $|V_l|=|V_r|=7$, $|U_r|=D-b+2$, and $|U_d|=b$. We refer to the edges incident to $U_d$ as {\it pendant}. See Figure \ref{fig:F(b,D)}.

As a part of a larger construction, Sevastianov \cite{S} proved the interval coloring properties of $F$.  Here we include the arguments for completeness.
The following lemma claims that the interval colorings of $F$ are very rigid. Depending on the smallest label used on $F$, the colors of certain edges are fixed. For integers $q, i, j$, $i\leq j$, we shall denote the interval of integers  $\{q+ i, q+i+1, \ldots,  q+j\} $ as $q+[i,j]$ and call it a {\it shift} of an interval $[i,j]$.  Moreover, for $i\leq j$ let $-[i,j]= [-j,-i]$.\\

\begin{lemma}[Sevastianov \cite{S}]\label{s}
For any positive integers $ b $ and $D$, where $b\leq D$ and $D$ is even,  the graph $F=F(b, T)$ is planar, bipartite, and interval colorable  for $T=D+25$. Moreover, for any interval coloring $c$ of $F$ the following properties hold:
\begin{enumerate}
\item{} $c(F)=c_1+ [0,T]$, for some integer $c_1$,\\
\item{} $c(w_lv_l)\in \{c_1+8, c_1+T-8\}$,\\
\item{}  if $c(w_lv_l)=c_1+8$, then the set of colors on the pendant edges is $c_1+11+[1, b]$; 
     if $c(w_lv_l)=c_1+T-8$, then set of colors on the pendant edges is $c_1+T-11-[1,  b].$  
%\item{} If $b=1$, i.e., there is one pendant edge in $F$, it gets a color $c_1+12$ or $c_1+T-12$.  Equivalently, if $b=1$ and the color of the pendant edge is $q$, then $c(F) = q-12+ [0,T]$ or $c(F)= q-T+12 +[0,T]$.  
\end{enumerate}      
\end{lemma}	

\begin{remark}
	The lemma implies that in any interval coloring of $F=F(b, T)$ the number of colors used is $T+1$, the colors of the pendant edges form an interval of $b$ numbers either starting with the first $13$th number used on $F$ or ending with the last $13$th number used on $F$. For example, if $T=37$, $b=2$, and for an interval coloring $c$ of $F$, $c(F)=\{3, \ldots, 40\}$, then the set of  colors of the pendant edges is either $\{15, 16\}$  or $\{27, 28\}$.   
\end{remark}

\begin{proof}
To see that $F$ is interval colorable, one can give an explicit coloring $c$ as follows. We denote by $c(U,z)$ a set of colors on edges incident to a set of vertices $U$ and a vertex $z$.  
Let  $c(v_lw_l)=9$, $c(w_lv')=10$, $c(w_lx)=11$, $c(xu)=12$, $c(uy)=D+14$, $c(yu_r)= D + 15$, $c(yw_r)=D+16$, $c(v'w_r)=D+17$, $c(v_rw_r)=D+18$, and
\begin{eqnarray*}
c(U_d, u) &= &\{13, 14,  \ldots, 12+b\},\\
c(V_l,v)&=&\{1, \ldots, 7\}, \\c(V_l,v_l)&=& \{2, \ldots, 8\},  \\
c(V_r,v)&=& \{D+26, D+25, \ldots, D+20\}, \\c(V_r,v_r)&= &\{D+25, D+24, \ldots, D+19 \}, \\
c(V_0, v)&= &\{8, 9, 10,  ~ 11, 12, \ldots,  D+15, D+16, ~  D+17, D+18,  D+19 \},\\
c(V_0,v')&=& \{7, 8, 9,~ ~ \, 12, 11,  \ldots,  D+16,  D+15, ~D+18, D+19, D+20 \}, \\  
c(U_r, u)&= & \{D+14, D+13, \ldots, 13+b\},\\  c(U_r,u_r)&=& \{D+15, D+14, \ldots, 14+b \}.
\end{eqnarray*}  
Note that for this iterative pattern at $V_0$, one needs $D$  to be even.\\~\\

The main idea of the remaining proof is an observation that  in an interval coloring of a graph the difference between the labels on two edges incident to a vertex $z$  is less than the degree of  $z$.
Note that the degree of the vertex $v$ in $F$ is $D+26$. Assume first that there is an edge $e$ labeled $1$ incident to $v$ and that $1$ is the smallest label at $v$. Then there is an edge $e'$ incident to $v$ and labeled $D+26$.\\

\begin{claim*}
	Either $e$ is incident to $V_l$ and $e'$ is incident to $V_r$ or $e'$ is incident to $V_l$ and $e$ is incident to $V_r$.
\end{claim*}

To prove the claim, note first that  $e$ and $e'$ can not both be incident to the same set $V_0, V_l$, or $V_r$.  Indeed, otherwise we consider two edges $e_1$ and $e_1'$ adjacent to $e$ and $e'$ respectively and incident to a vertex $r$ that is  $v', v_l$, or $v_r$, respectively.  Then the labels of $e_1$ and $e_1'$ are at most $2$ and at least $D+25$, respectively, contradicting the fact that the degree of $r\leq D+14$.
  
Now,  assume  that $e$ is incident to $V_0$ and  $e'$ is incident to $V_l$.  Consider a shortest path $P$ joining non-common endpoints of $e'$ and $e$ and avoiding $v$.  It has length $4$ and passes through vertices of degrees $2, 8, 3, D+14, $ and $2$, respectively. The respective labels on the edges of $P$ are at least $D+25$, $D+25 - 7, D+25 - 7 -2,$ and $ D+25-7-2-(D+13)$, respectively. The label on the edge of $P$ incident to $e$ is at least $3$, a contradiction since it must be at most $2$. By a similar argument, it is impossible for $e$ and $e'$ to be incident to $V_0$ and $V_r$,  to $V_l$ and $V_0$, or to $V_r$ and $V_0$, respectively.
This proves the claim.\\
 
Assume first that $e$ is incident to $V_l$ and $e'$ is incident to $V_r$. \\

To prove part 2 of the lemma, consider a path $P'$ joining non-common endpoints of $e'$ and $e$ and passing through $u$. Recall that $\cd(u)=D+3$.
It has $8$ edges, with second and next to last  edges being $v_lw_l$ and $w_rv_r$ respectively.  As before, the labels on consecutive edges of $P'$ have labels at most $2, 9, 11, 12, 14+D, 16+D, 18+D, 25+D$, respectively. Since the label of the last edge is exactly $D+25$, we see that all the edges of $P'$ have exactly the labels listed: $2, 9, 11, 12, 14+D, 16+D, 18+D, 25+D$. So, $c(w_lv_l)=9$ and $c(w_rv_r)=D+18$. \\

To prove part 3 of the lemma, consider vertices $x$ and $y$. Because of the properties of $P'$ we see that $c(yu_r)=15+D$. Since $c(xu)=12$, $c(uy)=14+D$, and $\cd (u)=D+3$, all other labels on edges incident to $u$ are greater than $12$ and less than $14+D$. As the labels on edges incident to $u_r$ form an interval, the largest label on an edge incident to $u_r$ is therefore $15+D$, such that the interval is $15 + D, 14+D, \ldots, 14 + b$. 
This implies that labels on edges incident to $u$ and $U_r$ form an interval of $D-b+2$ integers with the largest one $14+D$. 
Thus, edges incident to $U_d$ get labels forming an interval of $b$ integers with the smallest integer in the interval equal to $13$.\\

Finally, we have seen that $c(F)=\{1, \ldots, D+26\}$.\\

If $e$ is incident to $V_r$, a similar to above argument gives that the edges incident to $U_d$ have labels forming an interval of $b$ integers ending with $14+D$. 
In this case we have the roles of $w_r, v_r$ and $w_l, v_l$ swapped, so $c(w_lv_l)= D+18$ and $c(w_rv_r)=9$. As before, $c(F)=\{1, \ldots, D+26\}$. So, this established the lemma in case when $c_1=1$. \\

If $\min c(V_l \cup V_r \cup V_0, v) = c_1 \neq 1$, consider an interval coloring $c'$ defined by $c'(z)=c(z)-c_1+1$, $z\in V(F)$, i.e., done by an appropriate  label shift. 
Now, we have that $\min c'(V_l \cup V_r \cup V_0, v)=1$ and we can apply the above considerations. 
\end{proof}

\subsection{Construction and properties of the graph ${\bf F}(k,d)$}

For positive integers $k, d$ where $d$ is even and $d \geq 24$,  
let  $F_0=F( k, 3k^2d + 1)$. Note that there are $k$ pendant edges in $F_0$.
Further, let $F_j= F(1, 2jdk + 1)$, $j=1, \ldots, k$.  Note that each $F_j$ has a single pendant edge, $j=1, \ldots, k$.
Let ${\bf F}(k,d)$ be formed by first considering pairwise vertex disjoing copies of $F_0, F_1, \ldots, F_k$ and then identifying the $j$th pendant edge of $F_0$ with a pendant edge of $F_j$ such that the vertex of degree one of the $j$th pendant edge of $F_0$ is identified with the vertex of degree greater than one in the pendant edge of $F_j$, $j=1, \ldots, k$.  See Figure \ref{fig:F} for an illustration. 

\begin{figure}[hb!]
	\centering
	\includegraphics[width = 0.7 \textwidth]{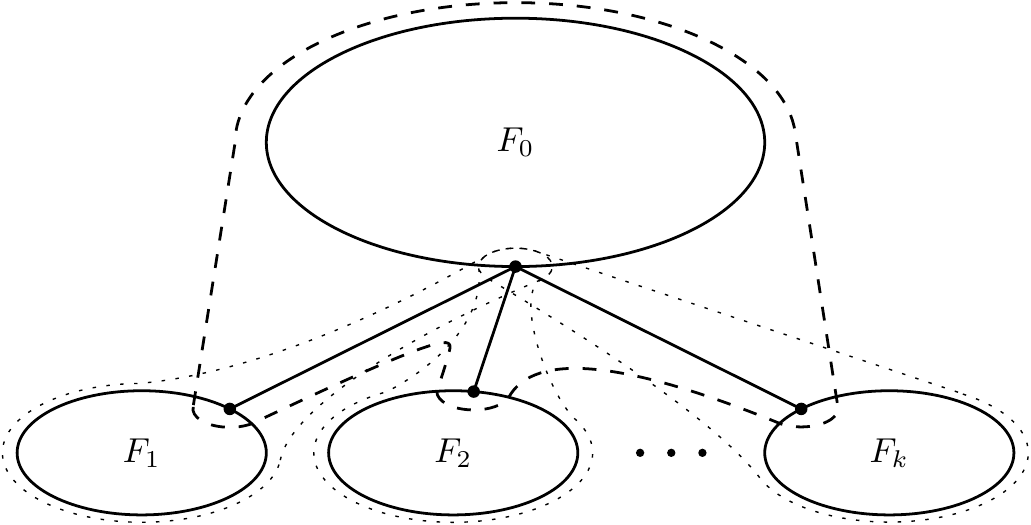}
	\caption{Graph ${\bf F}(k,d)$}
	\label{fig:F}
\end{figure}

\begin{lemma}\label{F}
For any positive integers  $k, d$, $d\geq 24$,  the graph ${\bf F} = {\bf F}(k,d) $ is interval colorable and has exactly $k$ gaps, each of size at least $d$, in its spectrum.
\end{lemma}

\begin{proof}
Since the $F_j$'s are interval colorable, one can create an interval coloring of each $F_j$, $j=1, \ldots, k$,  such that the pendant edge gets an arbitrary assigned value by shifting the labels appropriately. So, consider an arbitrary interval coloring of $F_0$ and then consider interval colorings of $F_1, \ldots, F_k$ such that the colors of the pendant edges equal to the colors of the corresponding pendant edges of $F_0$.  This gives an interval coloring of ${\bf F}$. \\

Let $T_0+1, T_1+1, \ldots, T_k+1$ be the number of colors used in  interval colorings of $F_0, \ldots, F_k$, respectively.
Let $c$ be an interval coloring of ${\bf F}$, assume without loss of generality that $c(F_0) = \{1, \ldots, T_0+1\}$.
Note that $c$ restricted to respective copies of  $F_0, 
F_1, \ldots, F_k$ is an interval coloring. Instead of saying ``a copy of $F_j$'', we just say $F_j$ for the rest of the proof.  
We know from Lemma \ref{s} that there are $k$ pendant edges in $F_0$ whose set of  colors  is either $\{13, 14, \ldots, 12+k\}$ or $\{(T_0 + 1)-12, (T_0 + 1)-13,  \ldots, (T_0 + 1)-11-k\}$. \\

Assume that the pendant edges of $F_0$ get the colors $13, 14, \ldots, 12+k$.  The situation when pendant edges in $F_0$ get colors $(T_0+1)-12, (T_0+1)-13,  \ldots, (T_0+1)-11-k$ is completely symmetric  resulting in the same number of colors used on ${\bf F}$ as in the respective configuration when the pendant edges of $F_0$ get the colors $13, 14, \ldots, 12+k$. Lemma \ref{s} implies that for each $j=1, \ldots, k$, the pendant edge of $F_j$ either gets the $13$th color of $c(F_j)$ or the last $13$th such color. We say that in the former case $F_j$  is of {\it type 1} under $c$ and in the latter case $F_j$ is of {\it type 2} under $c$.\\

If $F_j$ is  of type $1$ under $c$, $c(F_j) \subseteq c(F_0)$ since $T_0>T_j+k$.
If $F_j$ is of type $2$ under $c$, $\max c(F_j) \in c(F_0)$ and   $\min c(F_j)$ must take one of the values on the interval
 $[-T_j + 25, -T_j+24 + k]$, depending on whether the pendant edge of $F_j$ is identified with the edge of color $13, 14, \ldots,$ or $12+k$, respectively. \\

If  $j \in [1, k]$ is the largest index for which $F_j$ is of type $2$, then any such interval coloring of ${\bf F}$ uses $t$ colors for $t\in [T_0 + T_j - 22-k, T_0 + T_j - 23]$. Moreover, for any $t$ in this interval there is a corresponding interval coloring of ${\bf F}$. 
Observe that for any $j \in [1, k]$, there is an interval coloring of ${\bf F}$ such that $F_j$ is of type $2$ and each $F_i$ is of type $1$ for $i\in [1,k]\setminus \{j\}$.
 Thus, $S({\bf F})=\{T_0+1\} \cup \bigcup_{j=1}^k [T_0 + T_j - 22-k, T_0 + T_j - 23]$. Since $T_j =2jdk + 1$, $j\in [1,k]$,   we see that the interval spectrum has $k$ gaps of sizes at least $2dk - k - 23 \geq d$. 
\end{proof}

%%%%%%%%%%%%%%%%%%%%%%%%%%%%%%%%%%%%%%%%%%%%%%%%%%%%%%%%%%%%%%%%%
%%%%%%%%%%%%%%%%%%%%%%%%%%%%%%%%%%%%%%%%%%%%%%%%%%%%%%%%%%%%%%%%%
	
          \section{Proofs of the main results}\label{proofs}

%%%%%%%%%%%%%%%%%%%%%%%%%%%%%%%%%%%%%%%%%%%%%%%%%%%%%%%%%%%%%%%%%
%%%%%%%%%%%%%%%%%%%%%%%%%%%%%%%%%%%%%%%%%%%%%%%%%%%%%%%%%%%%%%%%%
	
	\begin{proof}[Proof of Theorem \ref{reg-bound}]
	Let $0<\gamma<1/2$ and $M$ be the constant guaranteed by Theorem \ref{gamma}. Let $G$ be a graph on $n$ vertices. Then, by Theorem \ref{gamma}, $G$ is a union of two graphs $G'$ and $G''$, where $\theta(G')\leq M$ and 
	$\| G''\| \leq \gamma n^2$. 
	By a result by Dean, Hutchinson, and Scheinermann \cite{DHS}, the arboricity of any graph is at most $\ceil{ \sqrt{e/2}}$, where $e$ is the number of edges in that graph. Since the interval thickness is at most the arboricity, we have $\theta(G'')\leq \sqrt{\gamma} n.$
	In particular, for large enough $n$, we have that $\theta(G) \leq M+ \sqrt{\gamma} n \leq 2\sqrt{\gamma} n.$
	This implies in particular that $\theta(n) =o(n)$.
	\end{proof}

\begin{proof}[Proof of Theorem \ref{spectral-gaps}]
This theorem follows immediately from Lemma \ref{F} using a construction of the graph ${\bf F}(k,d)$.
\end{proof}

\vskip 1cm

\noindent
{\bf Acknowledgements} ~~The authors would like to thank S. V. Avgustinovich for providing a copy of \cite{S} for them. The research was partially supported by the DFG grant FKZ AX 93/2-1.

\end{document}